\documentclass[12pt, reqno]{amsart}
\usepackage{amsfonts}
\usepackage{graphicx}
\usepackage{amsmath, amscd, amssymb}
\usepackage[latin5]{inputenc}

\newtheorem{theorem}{Theorem}
\theoremstyle{plain}

\newtheorem{lemma}{Lemma}

\newtheorem{remark}{Remark}

\numberwithin{equation}{section}

\begin{document}

\begin{center}
\thispagestyle{empty}
\markboth{\textbf{ }}
{\textbf{Some analysis on a fractional differential equation}}

\noindent {\Large \textbf{Some analysis on a fractional differential equation with discontinuous right-hand side \\[0pt]}}

\bigskip

\textbf{M\"ufit \c{S}AN$^{1},$ U\u{g}ur SERT$^{2}$}
\end{center}

\vspace{.03cm}

\begin{center}
\noindent \textit{ 1 Department of Mathematics, Faculty of Science, Çank{\i}r{\i} Karatekin University, Tr-18100, Çank{\i}r{\i}, Turkey\\
e-mail: mufitsan@karatekin.edu.tr} \\

\textit{2 Department of Mathematics, Faculty of Science, Hacettepe University, Beytepe, Tr-06800,  Ankara, Turkey, \\ e-mail: usert@hacettepe.edu.tr}
\end{center}

\medskip

\vspace{.4cm}

\noindent\textbf{Abstract} In this article, we consider an initial
value problem for a nonlinear differential equation with
Riemann-Liouville fractional derivative. By proposing a new
approach, we prove local existence and uniqueness of the solution
when the nonlinear function on the right hand side of the equation under consideration is not continuous on $[0,T]\times\mathbb{R}.$ \\


\medskip

\bigskip

\section{\noindent Introduction and Motivation}

The birth of fractional calculus dates back to the early days of
differential calculus. The number of researchers working on
fractional calculus were inadequate when compared to researchers
in differential calculus studies, eventually there had been no
progress on this field in a reasonable amount of time. However,
the interest in fractional calculus and fractional differential
equations has increased considerably for the last three decades.
This has lead to a rapidly development in fractional calculus by
virtue of the techniques, methods and results used in the ordinary
differential calculus. Evidently, a substantial part of the
interest in this subject derives from initial-value problems
(IVPs) and boundary-value problems (BVPs) for the fractional
differential equations with fractional derivatives such as
Riemann-Liouville (RL), Caputo, Caputo-Fabrizio,
Gr\"{u}nwald-Letkinov etc. Existence and uniqueness of solution
for the IVPs and BVPs were investigated by many researchers (see
for example \cite{Bai}, \cite{Del}, \cite{Diethelm2}, \cite{Fer},
\cite{Heymans}-\cite{Laks2}, \cite{Samko}-\cite{Zhang}). These
articles deals with the qualitative properties of equations with
continuous right-hand side. However, in this paper, we address an
equation with Riemann-Liouville derivative which has discontinuous
right-hand side. Considering the coincidence of first order RL
derivative with ordinary derivative (see \cite{Pod}-\cite{Samko}),
real-world applications of an equation with RL derivative which
has discontinuous right-hand side would be enlightening from many
problems arising from mechanics, electrical engineering and the
theory of automatic control. Differential equations with
discontinuous right-hand sides, in particular with a function
$f(x,t)$ which is discontinuous in $x$ and continuous in $t$ were
studied widely in the literature. For these studies, we can refer
the book of Filippov \cite{Filip}(also the references cited
therein), which is accepted as a basic source for the general
theory of discontinuous dynamical systems.

In this paper, we investigate the following initial value problem
for a differential equation with RL fractional derivative:
\begin{equation*}
\begin{cases}
&D^{a}u(x) = f\big(x,u(x)\big), \ \ x>0   \tag{1.1}  \\
&u(0)=u_{0},\ \       \\
\end{cases}
\end{equation*}
where $0<a<1$ (is valid throughout the paper), $u_{0}\neq 0$ is a
real number and the function $f$ will be specified later. The
operator, $D^{a}$ represents RL fractional derivative of order
$a,$ which is defined by combining the ordinary derivative and RL
fractional integral $I^{a}$ as follows:
\begin{align*}
D^{a}u\left(x\right):= \frac{d}{dx} \left[_{x}I^{1-a}u\right] \ \ \text{with} \  \ I^{a}u\left(x\right):=\frac{1}{\Gamma \left(a\right)}
\int_{0}^{x}\frac{u\left(\xi\right)}{\left(x-\xi\right)^{1-a}}d\xi.
\end{align*}
Here, $\Gamma(.)$ is the well-known Gamma function.
\par Problem (1.1) with continuous right-hand side was first discussed in
\cite{Laks1} and it was claimed that the continuous solution of
the problem exists on the interval $[0,T]$. Nevertheless, Zhang
\cite{Zhang} gave an example which indicates that the initial
condition $u(0)=u_{0}$ (except $u_{0}=0$) is not suitable for
investigating the existence of continuous solution of (1.1),
whenever the function $f$ is continuous on
$[0,T]\times\mathbb{R}.$ Accordingly, \c{S}an \cite{San2} considered problem (1.1) with $f$ which satisfies the following conditions:\\

(1.2) $f(x,t)$ and $x^{a}f(x,t)$ are continuous on $(0,T]\times\mathbb{R}$ and $[0,T]\times\mathbb{R},$ respectively,

(1.3) $x^{a}f(x,u_{0})\big|_{x=0}=u_{0}/\Gamma(1-a).$ \\

\noindent In \cite{San2}, it is proved that the condition (1.3) is
necessary for the existence of the continuous solution of problem
(1.1). Author also gave a partial answer to the question of the
existence of continuous solutions of (1.1). Problem (1.1)
represents a system and the initial condition must be independent
of tools we analyze. Based on this view, after a discussion with
Manuel D. Ortigueira (see also \cite{ort1},\cite{ort2}) we draw a
conclusion that it would be more accurate to discuss the
nonexistence of a continuous solution of (1.1) instead of querying
the suitability of the initial condition. In fact, if there were a
continuous solution $u$ of (1.1) when $f$ is continuous on
$[0,T]\times\mathbb{R},$ then by using the compositional relation
$u\left(x\right)=I^{a}D^{a}u(x)$ proved in Proposition 2.4 in
\cite{Del} and  $u\in C[0,T],$ $D ^{a}u\in C[0,T],$ it would be
shown that
\begin{align*}
u\left(x\right)=\frac{1}{\Gamma \left(a\right)}\int_{0}^{x}
\frac{\ f\left(\xi ,u\left(\xi\right)\right)}{\left(x-\xi\right)
^{1-a}}d\xi, \ \ \ x\in\left[0,T\right]. \tag{1.4}
\end{align*}
By the continuity of $f(x,u(x))$ on $\left[0,T\right]$ we obtain,

$$0\neq u_{0}=\lim_{x\to 0^{+}}u(x)=\frac{1}{\Gamma(a)}\int^{1}_{0}\frac{\lim_{x\to 0^{+}}x^{a}f(xt,u(xt))}{(1-t)^{1-a}}d\xi=0,$$
which is a contradiction. This implies that problem (1.1) with
continuous right-hand side and initial condition $u(0)=u_{0}\neq
0$ does not have a continuous solution $u(x)$.  \\
In the following section, under the conditions (1.2)-(1.3), we
first prove the existence of continuous solutions of (1.1) by
using Leray-Schauder alternative. However, as seen in the sequel,
this theorem does not inform us about the existence interval of
the solution. In the literature, there are some researchers (e.g.
\cite{Bal}, \cite{Eloe} and references therein) who worked on the
existence interval and maximal existence interval of the solution
for some fractional differential equations. For example, Mustafa
and Baleanu \cite{Bal} presented a method together with
Leray-Schauder alternative to obtain a better estimate for the
existence interval of the continuous solution of a problem they
considered. Such an approach cannot be applied directly to the
analysis of (1.1) hence, developing a new technique and using
Schauder's fixed point theorem, we attain a Peano-type existence
theorem for (1.1) which explicitly shows the existence interval of the solution. \\
For the uniqueness of problems similar to (1.1), there exist some
Nagumo-type uniqueness results (see, for example, \cite
{Diethelm2}, \cite{Fer}, \cite{Laks2}, \cite{Sin}) which were
proved by the technique and approach developed by Diaz
\cite{Diaz}. Besides the Nagumo-type conditions, for functions
satisfying certain conditions Diethelm \cite{Diethelm2} and Odibat
\cite{Odibat} employed some fractional mean value theorems to
establish uniqueness results.  We verify the existence and
uniqueness of the continuous solution of (1.1) when the function
$f$ additionally fulfills a Nagumo-type condition which is similar
to the Lipschitz-type condition used in Theorem 3.5 in \cite{Del}.
Here, we present a novel technique combining with a fractional
mean value theorem for functions $u\in C([0,T])$ satisfying
$D^{a}u\in C(0,T]$ and $x^{a}D^{a}u\in C[0,T]$ to prove the
uniqueness of (1.1).

\section {Preliminaries and Main Results}

Before proceeding to investigate problem (1.1), we remind some
basic facts from functional analysis which are required for the
main results. At first, we give a lemma which shows the
equivalence of the solutions of problem (1.1) and solutions of the
integral equation (1.4) \cite{San2}.

\begin{lemma}  Under conditions (1.2)-(1.3), the continuous solutions of (1.1) are also the solutions
of the integral equation (1.4), vice versa.
\end{lemma}

Let us define the operator $\mathcal{M}:C([0,T])\mapsto C([0,T])$
associated with integral equation (1.4) as follows:
\begin{align*}
\mathcal{M}u\left(x\right):=\frac{1}{\Gamma
\left(a\right)}\int_{0}^{x} \frac{\ f\left(\xi
,u\left(\xi\right)\right)}{\left(x-\xi\right) ^{1-a}}d\xi, \ \ \
x\in\left[0,T\right]. \tag{2.1}
\end{align*}

Since the fixed points of operator $\mathcal{M}$ coincide with the
solutions of integral equation (1.4), our goal is to find out the
fixed points of operator $\mathcal{M}$ by applying the following
fixed point theorems \cite{Deim},\cite{Drabek},\cite{Dug}:

\begin{theorem} (Schauder's fixed-point theorem) Let $X$ be a real Banach space, $B\subset X$ nonempty closed bounded
and convex, $\mathcal{M}: B\to B$ compact. Then, $\mathcal{M}$ has
a fixed point.
\end{theorem}

\begin{remark} In applications, it is usually too difficult or impossible to establish a set $B$ so that
$\mathcal{M}$ takes $B$ back into $B$ (see, \cite{Drabek}).
Therefore, it will be available to consider operator $\mathcal{M}$
that maps the whole space $X$ into $X$ to overcome this
difficulty. The following result is intimately associated with
what we stated above.
\end{remark}

\begin{theorem} (Leray-Schauder alternative). Let $X$ be
normed linear space and $\mathcal{M}:X\to X$ be a completely
continuous (compact) operator. Then, either there exists $u\in X$ such that
$$u=\mathcal{M}u$$ or the set
$$\mathcal{E}(\mathcal{M}):=\left\{u\in X: u=\mu\mathcal{M}(u) \ \text{for a certain} \ \mu\in (0,1) \right\} \eqno{(2.2)}$$
is unbounded.
\end{theorem}

The compactness of operator $\mathcal{M}$ was proved previously by
Theorem 2.5 in \cite{San2}, therefore it is sufficient to show
that remaining conditions of the fixed point theorems given above
will be fulfilled. The first existence theorem for problem (1.1)
is as follows:

\begin{theorem} Let conditions (1.2) and (1.3) be satisfied. Then, there exists at least one continuous solution $u\in C([0,T])$ of problem (1.1).
\end{theorem}

\begin{proof} We use Leray-Schauder alternative and it is sufficient to show that the set $\mathcal{E}(\mathcal{M})$ defined in (2.2) is bounded.
For an arbitrary $u\in\mathcal{E}(\mathcal{M})$ one has
\begin{align*}
\left|u(x)\right|&\leq\mu\frac{1}{\Gamma\left(a\right)}\left|\int_{0}^{x} \frac{f\left(\xi
,u\left(\xi\right)\right)}{\left(x-\xi\right) ^{1-a}}d\xi\right| \\
&<\frac{1}{\Gamma\left(a\right)}\left|\int_{0}^{x} \frac{\ f\left(\xi,u\left(\xi\right)\right)-\xi^{-a}\frac{u_{0}}{\Gamma(1-a)}+\xi^{-a}\frac{u_{0}}{\Gamma(1-a)}}{\left(x-\xi\right) ^{1-a}}\right| \\
&\leq M\Gamma(1-a)+\frac{\left|u_{0}\right|}{\Gamma(1-a)\Gamma(a)}\int_{0}^{x} \frac{1}{\xi^{a}\left(x-\xi\right)^{1-a}}d\xi\\
&= M\Gamma(1-a)+\left|u_{0}\right|,
\end{align*}
where $M=\sup_{(x,t)\in[0,T]\times\mathbb{R}}\left|f(x,t)\right|.$
Thus, for any $u\in\mathcal{E}(\mathcal{M})$ we get
$$\sup_{x\in[0,T]}\left|u(x)\right|<M\Gamma(1-a)+\left|u_{0}\right|,$$
which yields that $\mathcal{E}(\mathcal{M})$ is bounded. As a
result of Leray-Scauder alternative, (1.1) admits at least one
solution in $C([0,T]).$
\end{proof}

Now, we give a mean value theorem for RL derivative to achieve the
existence and uniqueness results. For its proof, we follow the
path established in \cite{Diethelm2} and \cite{Odibat}.

\begin{lemma} Let $u\in C[0,T]$ with $D ^{a}u\in C(0,T]$ and $x^{a}D ^{a}u\in C[0,T]$ for $0<a<1.$
Then, there exists a function $\lambda=\lambda(x),$ $\lambda:[0,T]\to(0,T),$ $0<\lambda(x)<x$ such that
$$u(x)=\Gamma(1-a)(\lambda(x))^{a}D ^{a}u(\lambda(x)) \eqno{(2.3)}$$
is satisfied for all $x\in [0,T]$.
\end{lemma}

\begin{proof} Using equality $u\left(x\right)=I^{a}D^{a}u(x)$ and by mean value theorem of integral calculus we
have,
\begin{align*}
u(x)=&\frac{1}{\Gamma\left(a\right)}\int_{0}^{x}\frac{D^{a}u(\xi)}{\left(x-\xi\right)^{1-a}}d\xi \\
=&\frac{1}{\Gamma\left(a\right)}\int_{0}^{x}\frac{\xi^{a}D^{a}u(\xi)}{\xi^{a}\left(x-\xi\right)^{1-a}}d\xi \\
=&\frac{(\lambda(x))^{a}D^{a}u(\lambda(x))}{\Gamma\left(a\right)}\int_{0}^{x}\frac{1}{\xi^{a}\left(x-\xi\right)^{1-a}}d\xi \\
=&\Gamma\left(1-a\right)(\lambda(x))^{a}D^{a}u(\lambda(x)),
\end{align*}
where $\lambda=\lambda(x)\in(0,x)$ for all $x\in[0,T].$
\end{proof}

\noindent\textbf{Remark 3.} In Theorem 1 \cite{Odibat}, the
dependence of $\lambda$ on $x$ was not clearly expressed.
Essentially, $\lambda$ is generally a function of $x.$ To see
this, let $u(x)=1+x^{2}$, from Lemma 2 one has
\begin{align*}
1+x^{2}&=\Gamma(1-a)\lambda^{a}\left(\frac{\lambda^{-a}}{\Gamma(1-a)}+\frac{\lambda^{2-a}}{\Gamma(3-a)}\right)
=1+\frac{2\lambda^{2}}{(1-a)(2-a)}.
\end{align*}
From here,
$$\lambda=\sqrt{\frac{(1-a)(2-a)}{2}}x\in (0,x)$$
which yields that $\lambda$ is a function of $x.$\\

Before we give a Peano-type existence theorem for problem (1.1),
let us make some notes. We use Schauder fixed-point theorem to
prove the existence of the continuous solution of (1.1). For this
it is ample to verify only $\mathcal{M}:B\to B,$ where $B$ is an
appropriate closed convex ball of $C([0,T])$ which will be
constructed later. In the following theorem, we present the second
existence result for problem (1.1).

\begin{theorem}\label{Thm1}  Let (1.2) be satisfied and $r,$ $T$ be fixed positive real numbers. Moreover, suppose that there exists a positive real number  $M^{*}$ such that
$$\left|x^{a}f(x,t)-\frac{u_{0}}{\Gamma(1-a)}\right|\leq M^{*}\max \big(x,\frac{\left|t-u_{0}\right|}{r}\big)  \eqno{(2.3)}$$
holds for all $x\in[0,T]$ and for all $t\in\left[u_{0}-r,u_{0}+r\right].$ Then, (1.1) has at least one continuous solution on $[0,T_{0}],$ where
\begin{align*}
T_{0}:=
\begin{cases}
\frac{r}{M^{*}\Gamma(1-a)}&,  \ \text{if} \ \ \ M^{*}\Gamma(1-a)\geq r,\\
\smallskip
\ \ \ \ T&, \ \text{if} \ \   \ M^{*}\Gamma(1-a)\leq r.\\
\end{cases}
\  \   \  \  \  \
\end{align*}
\end{theorem}

\begin{proof}
Let us first construct an appropriate closed convex ball of $C([0,T])$ according to inequality (2.3). For this, it is assumed that $$\left|x^{a}f(x,t)-\frac{u_{0}}{\Gamma(1-a)}\right|\leq M^{*}x  \eqno{(2.4)}$$
is fulfilled for all $x\in[0,T]$ and for all $t\in\left[u_{0}-r,u_{0}+r\right].$ Depending on this inequality, set
$$B_{r}(u_{0})\equiv \big\{u\in C[0,T_{0}]:\sup_{x\in[0,T_{0}]} \left|u(x)-u_{0}\right|\leq r\big\} $$
with $M^{*}\Gamma(1-a)\geq r.$ Then, for any $u\in B_{r}(u_{0}),$ from (2.4) one has
\begin{align*}
\left|\mathcal{M}u\left(x\right)-u_{0}\right|&\leq\frac{1}{\Gamma\left(a\right)}\int_{0}^{x}\frac{\left|f\left( \xi ,u\left(\xi\right)\right)-\xi^{-a}\frac{u_{0}}{\Gamma(1-a)}\right|}{\left(x-\xi\right)^{1-a}}d\xi   \\
&=\frac{1}{\Gamma\left(a\right)}\int_{0}^{x}\frac{\left|\xi^{a}f\left( \xi ,u\left(\xi\right)\right)-\frac{u_{0}}{\Gamma(1-a)}\right|}{\xi^{a}\left(x-\xi\right)^{1-a}}d\xi  \\
&\leq\frac{M^{*}}{\Gamma\left(a\right)}\int_{0}^{x}\frac{\xi}{\xi^{a}\left(x-\xi\right)^{1-a}}d\xi \\
&\leq M^{*}\left|x\right|\Gamma(2-a),
\end{align*} by the inequality $\Gamma(2-a)<\Gamma(1-a)$ for all $a\in(0,1)$, we get
$$\left|\mathcal{M}u\left(x\right)-u_{0}\right|\leq M^{*}T_{0}\Gamma(1-a).$$ By this inequality and the definition of $T_{0},$ one obtains
$$\sup_{x\in[0,T_{0}]}\left|\mathcal{M}u\left(x\right)-u_{0}\right|<M^{*}\Gamma(1-a)T_{0}\leq r. \eqno{(2.5)}$$ On the other hand, if
$$\left|x^{a}f(x,t)-\frac{u_{0}}{\Gamma(1-a)}\right|\leq
\frac{M^{*}}{r} \left|t-u_{0}\right| \eqno{(2.6)}$$ holds for all
$x\in[0,T]$ and for all $t\in\left[u_{0}-r,u_{0}+r\right],$ then
set $B_{r}$ as
$$B_{r}(u_{0})\equiv\big\{u\in C[0,T]:\sup_{x\in[0,T]} \left|u(x)-u_{0}\right|\leq r\big\}$$
for $M^{*}\Gamma(1-a)\leq r.$ Then, using (2.6) one gets
\begin{align*}
\left|\mathcal{M}u\left(x\right)-u_{0}\right|&\leq\frac{1}{\Gamma\left(a\right)}\int_{0}^{x}\frac{\left|\xi^{a}f\left( \xi ,u\left(\xi\right)\right)-\frac{u_{0}}{\Gamma(1-a)}\right|}{\xi^{a}\left(x-\xi\right)^{1-a}}d\xi  \\
&\leq\frac{M^{*}}{r\Gamma\left(a\right)}\int_{0}^{x}\frac{\left|u(\xi)-u_{0}\right|}{\xi^{a}\left(x-\xi\right)^{1-a}}d\xi \\
&\leq \frac{M^{*}}{\Gamma(a)}\int_{0}^{x}\frac{1}{\xi^{a}\left(x-\xi\right)^{1-a}}d\xi.
\end{align*}
for any $u\in B_{r}$ and for all $x\in [0,T].$ Therefore,
$$\sup_{x\in[0,T]}\left|\mathcal{M}u\left(x\right)-u_{0}\right|\leq M^{*}\Gamma(1-a)\leq r. \eqno{(2.7)}$$
From (2.5) and (2.7), we attain $\mathcal{M}(B_{r}(u_{0}))\subset
B_{r}(u_{0})$ which completes the proof.
\end{proof}

\begin{theorem} (\textit{Existence and Uniqueness}) Let $0<a<1,$ $0<T<\infty,$ and the conditions (1.2), (1.3) be satisfied. Moreover, suppose that the inequality $$x^{a}\left|f(x,t_{1})-f(x,t_{2})\right|\leq\frac{1}{\Gamma(1-a)}\left|t_{1}-t_{2}\right|  \eqno{(2.8)}$$
holds for all $x\in[0,T]$ and for all $t_{1},t_{2}\in\mathbb{R}.$ Then (1.1) admits a unique continuous solution on $[0,T].$
\end{theorem}
\begin{proof}
We proved the existence of the solution in Theorem 1. Thus
for the uniqueness, we assume that (1.1) has two different
continuous solutions $u_{1}$ and $u_{2}.$ We initially assume
$\omega(x)\not\equiv 0,$ where
$$
\omega(x):=\begin{cases}
\left|u_{1}(x)-u_{2}(x)\right|,& x>0 \\
\ \ \ \ \ \ \ \ 0 \ \ \ \ \ \ \  \ \ \ \  ,& x=0
\end{cases}
$$ It is easily seen that $\omega(x)$ is nonnegative for all
$x\in[0,T]$ and continuous for all these $x$ values except $x=0.$
For the continuity of $\omega(x)$ at $x=0,$  by using (1.4),
variable substitution $\xi=xt$ and condition (1.2) respectively,
we have
\begin{align*}
0\leq\lim_{x\to 0^{+}}\omega(x)&=\lim_{x\to 0^{+}}\frac{1}{\Gamma(a)}\left|\int_{0}^{x}\frac{f\left(\xi ,u_{1}\left(\xi\right)\right)-f\left(\xi ,u_{2}\left(\xi\right)\right)}{\left(x-\xi\right) ^{1-a}}
d\xi\right| \\
&\leq \lim_{x\to 0^{+}}\frac{1}{\Gamma(a)}\left|\int_{0}^{1}\frac{\left(xt\right)^{a}\big[f\left(xt ,u_{1}\left(xt\right)\right)-f\left(\xi ,u_{2}\left(xt\right)\right)\big]}{t^{a}\left(1-t\right)^{1-a}}dt\right|\\
&\leq\frac{1}{\Gamma(a)}\int_{0}^{1}\frac{\lim_{x\to 0^{+}}\big|\left(xt\right)^{a}\big[f\left(xt ,u_{1}\left(xt\right)\right)-f\left(\xi ,u_{2}\left(xt\right)\right)\big]\big|}{t^{a}\left(1-t\right)^{1-a}}dt \\
&=0
\end{align*}
which simply means that $\lim_{x\to 0^{+}}\omega(x)=0=w(0).$ \\
It is obvious that there exists a $\lambda\in (0,T]$ such that
$\omega(\lambda)\neq 0,$ i.e. $\omega(\lambda)>0.$ By using Lemma
2 and inequality (2.8), we deduced that
\begin{align*}
0<\omega(\lambda)&=\left|u_{1}(\lambda)-u_{2}(\lambda)\right|\\
&=\Gamma(1-a)\left|\lambda^{a}_{*}D^{a}(u_{1}-u_{2})(\lambda_{*})\right| \\
&=\Gamma(1-a)\left|f\left(\lambda_{*},u_{1}\left(\lambda_{*}\right)\right)-f\left(\lambda_{*},u_{2}\left(\lambda_{*}\right)\right)\right|\\
&\leq\left|u_{1}(\lambda_{*})-u_{2}(\lambda_{*})\right|=\omega(\lambda_{*})
\end{align*}
for some $\lambda_{*}\in (0,\lambda).$  If we take the same
procedure into account for the point $\lambda_{*},$ then there
exists some $\lambda_{2}\in(0,\lambda_{*})$ such that
$0<\omega(\lambda)\leq\omega(\lambda_{*})\leq\omega(\lambda_{2}).$
Continuing in the same way,  we construct a sequence
$\left\{\lambda_{n}\right\}_{n=1}^{\infty}\subset [0,\lambda)$
with $\lambda_{1}=\lambda_{*}$ satisfying $\lambda_{n}\to 0$ and
$$0<\omega(\lambda)\leq\omega(\lambda_{1})\leq\omega(\lambda_{2})\leq...\leq\omega(\lambda_{n})\leq... \eqno{(2.9)}$$
On the other hand, since $\omega(x)$ is continuous at $x=0$ and
$\lambda_{n}\to 0,$ then $\omega(\lambda_{n})\to \omega(0)=0$ that
contradicts with (2.9). Hence $\omega(x)\equiv 0,$ namely IVP
(1.1) admits a unique continuous solution.
\end{proof}
It is interesting to note that there are some other techniques and
theorems (see, for example \cite{Diethelm2} and \cite{Sin}) that enable us to take a positive fixed real
number larger than Nagumo constant or an
arbitrary positive real number instead of Nagumo constant
so that IVPs admit a unique solution under the
Nagumo-type condition. However, the mentioned techniques could not
be applied to problem (1.1), namely there does not exist a larger
number than $\frac{1}{\Gamma(1-a)}$ or an arbitrary positive real
number instead of $\frac{1}{\Gamma(1-a)}$ to guarantee the
uniqueness of the continuous solution of (1.1). The following example may clearly express the foregoing discussion. \\

\noindent \textbf{Example 1.} Let us take $f_{\beta}(x,t):=\frac{\Gamma(\beta+1)}{\Gamma(1-a+\beta)}x^{-a}\left(t+k\right)$
where $k=\frac{\Gamma(\beta-a+1)-\Gamma(1+\beta)\Gamma(1-a)}{\Gamma(1+\beta)\Gamma(1-a)},$ $\beta>0$ and $u_{0}=1$ in problem (1.1).
It is clear that conditions (1.2) and (1.3) are satisfied. However, inequality (2.8) does not hold for $f_{\beta}(x,t)$,
since  $\frac{\Gamma(\beta+1)}{\Gamma(1-a+\beta)}$ replaces instead of $\frac{1}{\Gamma(1-a)}$ in (2.8)
and $\frac{\Gamma(\beta+1)}{\Gamma(1-a+\beta)}>\frac{1}{\Gamma(1-a)}$ for $\beta>0$ and $a\in(0,1).$
That is to say, the solution of (1.1) may not be unique. Indeed, (1.1) has infinitely many solutions $u(x)=cx+1,$ where $c$ is an arbitrary real number.
Furthermore, it is to be pointed out that
$\frac{\Gamma(\beta+1)}{\Gamma(1-a+\beta)}\to \frac{1}{\Gamma(1-a)}$ and $f_{\beta}(x,t)\to f(x,t)=\frac{x^{-a}t}{\Gamma(1-a)}$ when $\beta\to 0$ and that,
for $f(x,t)=\frac{x^{-a}t}{\Gamma(1-a)},$ (1.1) with $u_{0}=1$ admits a unique solution in the form $u(x)=1.$

\end{document}